\documentclass[a4paper]{amsart}%
\usepackage{amscd}
\usepackage{amsthm}
\usepackage{graphicx}
\usepackage{amsmath}
\usepackage{amsfonts}
\usepackage{amssymb}%
\providecommand{\U}[1]{\protect\rule{.1in}{.1in}}

\newtheorem{theorem}{Theorem}

\newtheorem{corollary}[theorem]{Corollary}

\newtheorem{lemma}[theorem]{Lemma}

\DeclareMathOperator{\sgn}{sgn}
\begin{document}

\title[A counter-example]{Multilinear fractional integral operators: \\ a counter-example}
\author{Pablo Rocha}
\address{Departamento de Matem\'atica, Universidad Nacional del Sur, Av. Alem 1253 - Bah\'{\i}a Blanca 8000, Buenos
Aires, Argentina.}
\email{pablo.rocha@uns.edu.ar}
\thanks{\textbf{2.010 Math. Subject Classification}: 42B20, 42B30.}

\begin{abstract} 
By means of a counter-example we show that the multilinear fractional operator $\mathcal{I}_{\gamma}$ ($1 < \gamma < 2$) is not bounded from $H^{1}(\mathbb{R}) \times H^{p}(\mathbb{R})$ into $H^{q}(\mathbb{R})$, for $0 < p \leq \gamma^{-1}$ and $\frac{1}{q}= 1 +\frac{1}{p} - \gamma$.
\end{abstract}
\maketitle

\section{Introduction}

Given positive integers $m, n$ and a real number $0 < \gamma < mn$, it is define the multilinear fractional operator $\mathcal{I}_{\gamma}$ by

\[
\mathcal{I}_{\gamma} (f_1, ..., f_m)(x) = \int_{(\mathbb{R}^{n})^{m}} \frac{f_1(y_1) \cdot \cdot \cdot f_m(y_m)}{(|x - y_1| + \cdot \cdot \cdot +|x-y_m|)^{nm - \gamma}} dy_1 \cdot \cdot \cdot dy_m, \,\,\,\,\, x \in \mathbb{R}^{n}.
\]

\

Y. Lin and S. Lu in \cite{Lin} proved Hardy space estimates for the multilinear fractional operator $\mathcal{I}_{\gamma}$. More precisely, they proved that if $0 < \gamma < n$, $0 < p_1, . . . , p_m, q \leq 1$, and $q$ such that $\frac{1}{q}= \frac{1}{p_1} + \cdot \cdot \cdot + \frac{1}{p_m} - \frac{\gamma}{n} > 0$, then
\[
\| \mathcal{I}_{\gamma}(f_1, ..., f_m) \|_{L^{q}} \leq C \| f_1 \|_{H^{p_1}} \cdot \cdot \cdot \| f_m \|_{H^{p_m}}.
\]
Recently, D. Cruz-Uribe, K. Moen and H. van Nguyen in \cite{Cruz-Uribe} generalized the result of Lin and Lu to weighted Hardy
spaces on the full range $0 < \gamma < nm$. 

The purpose of this note is to give a counter-example to show that the multilinear fractional operator $\mathcal{I}_{\gamma}$ is not bounded from a product of Hardy spaces into a Hardy space. For them, we consider $n=1$, $m=2$, $\gamma = \alpha +1$ with $0 < \alpha < 1$, so 
$1 < \gamma < 2$, $2 - \gamma = 1 - \alpha$ and the multilinear fractional operator $\mathcal{I}_{\alpha +1}$ in this case is given by
\[
\mathcal{I}_{\alpha +1}(f_1,f_2)(x) = \iint_{\mathbb{R}^{2}} \frac{f_1(s) f_2(t)}{(|x-s|+|x-t|)^{1-\alpha}} ds dt, \,\,\,\ x \in \mathbb{R}.
\]
We will prove that the operator $\mathcal{I}_{\alpha +1}$ is not bounded from $H^{1}(\mathbb{R}) \times H^{p}(\mathbb{R})$ into $H^{q}(\mathbb{R})$, for $0 < p \leq (\alpha + 1)^{-1}$ and $\frac{1}{q}= \frac{1}{p} - \alpha$. 

\

We briefly recall the definition of Hardy space on $\mathbb{R}^{n}$. The Hardy space $H^p(\mathbb{R}^n)$ (for $0 < p < \infty$) consists of tempered distributions $f \in \mathcal{S}'(\mathbb{R}^n)$ such that for some Schwartz function $\varphi$ with $\int \varphi = 1$, the maximal operator
$$
{\displaystyle (\mathcal{M}_{\varphi }f)(x)=\sup _{t>0}|(\varphi_{t}*f)(x)|}$$
is in $L^p(\mathbb{R}^n)$, where $\varphi_t(x):=\frac{1}{t^n}\varphi(\frac{x}{t})$.
In this case we define $\left\Vert  f \right\Vert_{H^{p}}:=\left\Vert \mathcal{M}_\varphi f \right\Vert _{p}$ as the $H^{p}$ ``norm". It can be shown that this definition does not depend on the choice of the function $\varphi$. 
For $1 < p < \infty$, it is well known that $H^{p}(\mathbb{R}^{n}) \cong L^{p}(\mathbb{R}^{n})$, $H^{1}(\mathbb{R}^{n}) \subset L^{1}(\mathbb{R}^{n})$ strictly, and for $0 < p < 1$ the spaces $H^{p}(\mathbb{R}^{n})$ and $L^{p}(\mathbb{R}^{n})$ are not comparable.

The following sentences hold in Hardy spaces $H^{p}(\mathbb{R}^{n})$ for $0 < p \leq 1$ (see pp. 128-129 in \cite{St}):

\

(S1) A bounded compactly supported function $f$ belongs to $H^{p}(\mathbb{R}^{n})$ if and only if it satisfies the moment conditions $\int x^{\beta} f(x) dx = 0$ for all $|\beta| \leq n (p^{-1} -1)$.

\

(S2) If $f \in L^{1}(\mathbb{R}^{n}) \cap H^{p}(\mathbb{R}^{n})$ then $\int  x^{\beta} f(x) dx = 0$, whenever $|\beta| \leq n (p^{-1} -1)$ and the function $x^{\beta} f(x)$ is in $L^{1}(\mathbb{R}^{n})$.

\

To obtain our result we will compute explicitly in Section 2 the Fourier transform of the kernel $\left(  |x-s|+|x-t| \right)^{\alpha - 1}$ in the $x$ variable, this allows us to get the following identity 
\[
\int_{\mathbb{R}} \mathcal{I}_{\alpha +1}(a_1, a_2)(x) dx = \frac{\alpha - 1}{\alpha} \iint_{\mathbb{R}^{2}} a_{1}(s) \, a_2(t) \, |t-s|^{\alpha} \, ds dt,
\]
valid for bounded functions $a_1$ and $a_2$ having compact support with $\int a_1 =0$ or $\int a_2 =0$. Then, from (S2), the counter-example will follow to consider $a_1 \in H^{1}(\mathbb{R})$ and $a_2 \in H^{p}(\mathbb{R})$ such that $\iint_{\mathbb{R}^{2}} a_{1}(s) \, a_2(t) \, |t-s|^{\alpha} \, ds dt \neq 0$.
 
\

\textbf{Notation:} We use the following convention for the Fourier transform in $\mathbb{R}$ $\widehat{f}(\xi) = \int f(x) e^{-ix \xi} \, dx$. As usual we denote with $\mathcal{S}(\mathbb{R})$ the Schwartz space on $\mathbb{R}$.

\section{Preliminaries}

We start with the following lemma.

\begin{lemma} For $0 < \alpha < 1$ and $s \neq t \in \mathbb{R}$ fixed, let $K_{s,t}^{\alpha}$ be the function defined in $\mathbb{R}$ by
\[
K_{s,t}^{\alpha}(x)= \left(  |x-s|+|x-t| \right)^{\alpha - 1}, \,\,\,\, x \in \mathbb{R}.
\]
Then 
\[
\widehat{K_{s,t}^{\alpha}}(\xi) = - 2^{\alpha} \, \Gamma(\alpha) \sin \left(\frac{(\alpha -1)\pi}{2} \right) e^{-i \xi(\frac{s+t}{2})}|\xi|^{-\alpha} + |t-s|^{\alpha - 1} \sgn{(t-s)} \int_{s}^{t} e^{-i x \xi} dx 
\]
\[
- \, \frac{|t-s|^{\alpha}}{\alpha} e^{-i \frac{(s+t)}{2} \xi} \cos \left(\frac{|t-s| \xi}{2} \right) \, + \, 
\frac{i \, 2^{\alpha} \, \xi e^{-i \xi(\frac{s+t}{2})}}{\alpha} 
\int_{0}^{\frac{|t-s|}{2}} x^{\alpha} \sin(x \xi) dx, 
\]
in the distributional sense.
\end{lemma}

\begin{proof} First we assume that $s < t$. Then for each $\phi \in \mathcal{S}(\mathbb{R})$ fixed, we have
\[
\left(\widehat{K_{s,t}^{\alpha}}, \phi \right)= \left(K_{s,t}^{\alpha}, \widehat{\phi} \, \right) = \int_{\mathbb{R}} K_{s,t}^{\alpha}(x) 
\widehat{\phi}(x) dx
\]
\[
= \int_{t}^{+\infty}  K_{s,t}^{\alpha}(x) \widehat{\phi}(x) dx + \int_{s}^{t}  K_{s,t}^{\alpha}(x) \widehat{\phi}(x) dx +
\int_{-\infty}^{s}  K_{s,t}^{\alpha}(x) \widehat{\phi}(x) dx = I + II + III. 
\]
Let us now proceed to compute each one of these integrals,
\[
I = \int_{\mathbb{R}} \chi_{(t, +\infty)}(x) \left( 2x -(s+t) \right)^{\alpha -1}
\, \widehat{\phi}(x) dx
\]
\[
= 2^{\alpha - 1} \int_{\mathbb{R}}  x^{\alpha -1} \chi_{(\frac{t-s}{2}, +\infty)}(x)
\, (e^{-i(\cdot) \frac{(s+t)}{2}}\phi) \,\, \widehat{} \,\, (x) dx
\]
\[
= 2^{\alpha - 1} \int_{\mathbb{R}} x_{+}^{\alpha - 1} (e^{-i(\cdot) \frac{(s+t)}{2}}\phi) \,\, \widehat{} \,\, (x) dx - 2^{\alpha - 1} \int_{\mathbb{R}} x^{\alpha - 1}
\chi_{(0, \frac{t-s}{2})}(x) (e^{-i(\cdot) \frac{(s+t)}{2}}\phi) \,\, \widehat{} \,\, (x) dx 
\]
\[
=2^{\alpha - 1} \int_{\mathbb{R}} x_{+}^{\alpha - 1} (e^{-i(\cdot) \frac{(s+t)}{2}}\phi) \,\, \widehat{} \,\, (x) dx 
\]
\[
- 2^{\alpha - 1} \int_{\mathbb{R}} \left( \frac{(t-s)^{\alpha} e^{-i \xi \frac{(t-s)}{2}}}{2^{\alpha} \alpha} + \frac{i \xi}{\alpha}
\int_{0}^{\frac{t-s}{2}} x^{\alpha} e^{-ix \xi} dx \right) e^{-i \xi \frac{(s+t)}{2}}\phi(\xi) d\xi,
\]
to compute $III$ we proceed as in $I$, thus
\[
III = 2^{\alpha - 1} \int_{\mathbb{R}} x_{-}^{\alpha - 1} (e^{-i(\cdot) \frac{(s+t)}{2}}\phi) \,\, \widehat{} \,\, (x) dx 
\]
\[
- 2^{\alpha - 1} \int_{\mathbb{R}} \left( \frac{(t-s)^{\alpha}}{2^{\alpha}\alpha}e^{i \xi (\frac{t-s}{2})} - \frac{i \xi}{\alpha}
\int_{0}^{\frac{t-s}{2}} x^{\alpha} e^{ix \xi} dx \right) e^{-i \xi \frac{(s+t)}{2}}\phi(\xi) d\xi,
\]
so
\[
I + III= 2^{\alpha - 1} \int_{\mathbb{R}} |x|^{\alpha - 1} (e^{-i(\cdot) \frac{(s+t)}{2}}\phi) \,\, \widehat{} \,\, (x) dx 
\]
\[
- \, \int_{\mathbb{R}} \frac{(t-s)^{\alpha}}{\alpha} e^{-i \frac{(s+t)}{2} \xi} \cos \left(\frac{(t-s) \xi}{2} \right) \phi(\xi) d\xi
\]
\[ 
+ \, \int_{\mathbb{R}} \left( \frac{i \, 2^{\alpha} \, \xi e^{-i \xi(\frac{s+t}{2})}}{\alpha} 
\int_{0}^{\frac{(t-s)}{2}} x^{\alpha} \sin(x \xi) dx \right) \phi(\xi) d\xi. 
\]
Now $II$ is easy, indeed 
\[
II = \int_{\mathbb{R}} \chi_{(s,t)}(x) (t-s)^{\alpha - 1} \, \widehat{\phi}(x) dx
= \int_{\mathbb{R}} \left( (t-s)^{\alpha - 1} \int_{s}^{t} e^{-ix \xi} dx \right) \phi(\xi) d\xi.
\]
Since
\[
\widehat{|x|^{\alpha -1}}(\xi) = -2 \Gamma(\alpha) \sin \left(\frac{(\alpha -1)\pi}{2} \right) |\xi|^{-\alpha}
\]
(see equation (12), pp. 173, in \cite{Gelfand}), the lemma follows for the case $s < t$. Finally, exchanging the roles of $s$ and $t$ we obtain the statement of the lemma.
\end{proof}

\begin{corollary} If $a_1$ and $a_2$ are two bounded functions on $\mathbb{R}$ with compact support and such that $\int a_1 = 0$ or $\int a_2 = 0$, then
\[
\int_{\mathbb{R}} \mathcal{I}_{\alpha +1}(a_1, a_2)(x) dx = \frac{\alpha - 1}{\alpha} \iint_{\mathbb{R}^{2}} a_{1}(s) \, a_2(t) \, |t-s|^{\alpha} \, ds dt.
\]
\end{corollary}

\begin{proof} It is easy to check that $\mathcal{I}_{\alpha +1}(a_1, a_2)(\cdot) \in L^{1}(\mathbb{R})$. Let $\varphi \in \mathcal{S}(\mathbb{R})$ be an {\bf even} function such that $\varphi(0)=1$ and for $\epsilon > 0$ let $\varphi_{\epsilon}(x) = \varphi(\epsilon x)$. Since
\[
\int_{\mathbb{R}} \mathcal{I}_{\alpha +1}(a_1, a_2)(x) dx = 
\lim_{\epsilon \to 0^{+}} \int_{\mathbb{R}} \mathcal{I}_{\alpha +1}(a_1, a_2)(x) \varphi_{\epsilon}(x)dx,
\]
we will proceed to compute this limit.
\[
\lim_{\epsilon \to 0^{+}} \int_{\mathbb{R}} \mathcal{I}_{\alpha +1}(a_1, a_2)(x) \varphi_{\epsilon}(x)dx =
\lim_{\epsilon \to 0^{+}} \iint_{\mathbb{R}^{2}} a_1(s) a_2(t) \left(\int_{\mathbb{R}} K_{s,t}^{\alpha}(x) \varphi_{\epsilon}(x)dx \right)dsdt
\]
\[
= \lim_{\epsilon \to 0^{+}} \iint_{\mathbb{R}^{2}} a_1(s) a_2(t) \left(\int_{\mathbb{R}} \widehat{K_{s,t}^{\alpha}}(\xi) 
\widehat{\varphi_{\epsilon}}(\xi)d\xi \right)dsdt
\]
\[
= \iint_{\mathbb{R}^{2}} a_1(s) a_2(t) \lim_{\epsilon \to 0^{+}} \left(\int_{\mathbb{R}} \left( \widehat{K_{s,t}^{\alpha}}(\epsilon \, \xi) + 
2^{\alpha} \, \Gamma(\alpha) \sin \left(\frac{(\alpha -1)\pi}{2} \right) |\epsilon \xi|^{-\alpha} \right) \widehat{\varphi}(\xi)d\xi \right)dsdt
\]
\[
= \frac{\alpha - 1}{\alpha} \iint_{\mathbb{R}^{2}} a_{1}(s) \, a_2(t) \, |t-s|^{\alpha} \, ds dt,
\]
where the third equality follows from the moment condition of $a_1$ (or $a_2$) and the last one from Lemma 1 and that $\varphi(0)=1$.
\end{proof}

\section{A Counter-example}

We take $a_1(s) = \chi_{(-1,0)}(s) - \chi_{(0,1)}(s)$ and $a_2(t) = a_1(t-2)$. From (S1) it follows that $a_1 \in H^{1}(\mathbb{R})$ and 
$a_2 \in H^{(\alpha +1)^{-1}}(\mathbb{R})$ for each $0 < \alpha <1$. A computation gives

\[
\iint a_1(s) a_2(t) |t-s|^{\alpha} ds dt = \frac{4 \cdot 3^{\alpha + 2} - 4^{\alpha +2} - 6 \cdot 2^{\alpha +2} +4}{(\alpha +1) (\alpha +2)} 
\neq 0.
\]
${}$\\
From (S2) and corollary 2 it obtains that $\mathcal{I}_{\alpha +1}(a_1, a_2)(\cdot) \notin H^{1}(\mathbb{R})$, for each $0 < \alpha <1$. 
For $0 < p < (\alpha+1)^{-1}$ and $\frac{1}{q} = \frac{1}{p} - \alpha$, we take $N$ as any fixed integer with $N > p^{-1}-1$, then the set of all bounded, compactly supported functions for which $\int x^{\beta} f(x) dx =0$, for all $|\beta| \leq N$ is dense in $H^{r}(\mathbb{R})$ for each $p \leq r \leq 1$ (see 5.2 b), pp. 128, in \cite{St}). In particular, there exists $b \in H^{p}(\mathbb{R})$ such that $\|a_1 \|_{H^{1}} \| a_2 - b \|_{H^{(\alpha+1)^{-1}}} < \left| \int_{\mathbb{R}} \mathcal{I}_{\alpha+1}(a_1, a_2) (x) dx \right|/2C$. Then
\[
\left| \int_{\mathbb{R}} \mathcal{I}_{\alpha+1}(a_1, b) (x) dx \right| \geq \left| \int_{\mathbb{R}} \mathcal{I}_{\alpha+1}(a_1, a_2) (x) dx \right| -  \int_{\mathbb{R}} | \mathcal{I}_{\alpha+1}(a_1, a_2 - b) (x) | dx 
\]
\[
\geq \left| \int_{\mathbb{R}} \mathcal{I}_{\alpha+1}(a_1, a_2) (x) dx \right| - C \|a_1 \|_{H^{1}} \| a_2 - b \|_{H^{(\alpha+1)^{-1}}} 
\]
\[
> \frac{1}{2}\left| \int_{\mathbb{R}} \mathcal{I}_{\alpha+1}(a_1, a_2) (x) dx \right| > 0
\]
where the second inequality follows from Theorem 1.1 in \cite{Cruz-Uribe} with $p_1 = 1$, $p_2 = (\alpha+1)^{-1}$ and $q=1$. But then the operator $\mathcal{I}_{\alpha + 1}$ is not bounded from $H^{1}(\mathbb{R}) \times H^{p}(\mathbb{R})$ into $H^{q}(\mathbb{R})$ for $0 < p \leq (\alpha + 1)^{-1}$ and $\frac{1}{q}= \frac{1}{p} - \alpha$, since $\int_{\mathbb{R}} \mathcal{I}_{\alpha+1}(a_1, b) (x) dx \neq 0$.

\qquad

We conclude this note by summarizing our main result in the following theorem.

\begin{theorem} For $1 < \gamma < 2$, let $\mathcal{I}_{\gamma}$ be the multilinear fractional integral operator given by
\[
\mathcal{I}_{\gamma}(f_1,f_2)(x) = \iint_{\mathbb{R}^{2}} \frac{f_1(s) f_2(t)}{(|x-s|+|x-t|)^{2-\gamma}} ds dt, \,\,\,\, x \in \mathbb{R}.
\]
Then the operator $\mathcal{I}_{\gamma}$ is not bounded from $H^{1}(\mathbb{R}) \times H^{p}(\mathbb{R})$ into $H^{q}(\mathbb{R})$ for \,\,\,\,\,\,\,\,\,\, $0 < p \leq \gamma^{-1}$ and $\frac{1}{q}= 1 +\frac{1}{p} - \gamma$.
\end{theorem}


\end{document}